\documentclass[a4paper]{article}
\usepackage{amsmath}
\usepackage{amsfonts}
\usepackage{amssymb}
\usepackage{graphicx}
\usepackage[utf8]{inputenc}
\usepackage[dvipsnames]{xcolor}
\usepackage[shortlabels]{enumitem}
\usepackage{hyperref}
\usepackage[utf8]{inputenc}
\usepackage{amssymb,amsmath,amsthm, amsfonts}
\usepackage{mathtools}
\usepackage[algo2e,linesnumbered,noend,noline]{algorithm2e}
\usepackage{algorithmic,refcount}
\usepackage[a4paper]{geometry}

\newtheorem{Corollary}{Corollary}
\newtheorem{Theorem}{Theorem}
\newtheorem{Lemma}{Lemma}
\newtheorem{Claim}{Claim}
\newtheorem{Proposition}{Proposition}

\newtheorem{case}{Case}

\newenvironment{proofof}[1]{\begin{proof}[Proof of #1]}{\end{proof}}

\title{Packing of mixed hyperarborescences with flexible roots\\
 via matroid intersection}
\author{Florian H\"{o}rsch and Zolt\'an Szigeti,\\ GSCOP, Grenoble}

\begin{document}

\maketitle

\begin{abstract}
Given a mixed hypergraph $\mathcal{F}=(V,\mathcal{A}\cup \mathcal{E})$, functions $f,g:V\rightarrow \mathbb{Z}_+$ and an integer $k$, a packing of $k$ spanning mixed hyperarborescences is called $(k,f,g)$-flexible if every $v \in V$ is the root of at least $f(v)$ and at most $g(v)$ of the mixed hyperarborescences. We give a characterization of the mixed hypergraphs admitting such packings. This generalizes results of Frank and, more recently, Gao and Yang. Our approach is based on matroid intersection, generalizing a construction of Edmonds. We also obtain an algorithm for finding a minimum weight solution to the above mentioned problem.
\end{abstract}

\section{Introduction}

The purpose of this article is to generalize a recent result of Gao and Yang on packing mixed arborescences with flexible roots from mixed graphs to mixed hypergraphs using matroid intersection. This also yields a weighted algorithm for the corresponding problem.
\medskip

In order to understand the introduction, the reader may find all the necessary definitions in Section \ref{deflab}.
\medskip

The most basic setting when dealing with arborescence packings is the following one: Given a directed graph $D=(V,A)$ and a multiset $R$ of vertices in $V$, we want to find a packing $\{B_r:r \in R\}$ of spanning $r$-arborescences. The following result was proven by Edmonds in 1973 in \cite{e} and is fundemental to the theory of arborescence packings. It gives a complete characterization for the existence of packings of spanning arborescences with fixed roots in the basic setting.

\begin{Theorem}\label{basic1}
Let $D=(V,A)$ be a digraph and $R$ a multiset of $V.$ There exists a packing $\{B_r:r \in R\}$ of spanning $r$-arborescences  in $D$ if and only if 
$d_A^{-}(X)\geq |R-X| \text{ for all } \emptyset \neq X\subseteq V.$
\end{Theorem}

Another celebrated achievement of Edmonds is a connection between the theory of spanning arborescence packings in digraphs and matroid theory. For that, he considers two matroids on the arc set  $A$ of $D$. The first matroid ${\sf M}_1$ is the $k$-sum of the forest matroid of the underlying graph of $D$ where $k=|R|$. In other words, an arc set is independent in ${\sf M}_1$ if the corresponding edge set can be partitioned into $k$ forests. The second matroid ${\sf M}_2$ is the direct sum of the uniform matroids of rank $k-|R\cap v|$ on the set $\delta_A^-(v)$ of arcs entering $v$ for all $v \in V$. We refer to these matroids as the {\it $k$-forest matroid} and the {\it $(k,R)$-partition matroid} of $D$, respectively. Edmonds proved in \cite{e4} that the arc sets of a packing of spanning arborescences with respect to the root set $R$ are exactly the common independent sets of ${\sf M}_1$ and ${\sf M}_2$. The following observation which can simply be obtained from Theorem \ref{basic1} is crucial for this modeling. It can be found as Theorem 13.3.20 in \cite{book}. It can also be found as Lemma 5.4.6 in \cite{notes}, where a direct proof is provided.
\begin{Theorem}\label{basic2}
Let $D=(V,A)$ be a digraph and $R$ a multiset of $V$ of size $k.$ 
 Some $A'\subseteq A$ is the arc set of a packing $\{B_r:r \in R\}$ of spanning $r$-arborescences in $D$ if and only if the underlying edge set of $A'$ is the edge set of a packing of $k$ spanning trees and $d_{A'}^-(v)=k-|R \cap v|$ for all $v \in V$. 
\end{Theorem}

A first way of generalizing the results obtained in the basic setting is the consideration of directed hypergraphs instead of digraphs. The following generalization of Theorem \ref{basic1} was proved in a stronger form by Frank, Kir\'aly and Kir\'aly in \cite{fkk}.
\begin{Theorem}\label{hyperedmonds1}
Let $\mathcal{D}=(V,\mathcal{A})$ be a dypergraph and $R$ a multiset in $V.$ There exists a packing $\{\mathcal{B}_r:r \in R\}$ of spanning $r$-hyperarborescences in $\mathcal{D}$   if and only if $d_{\mathcal{A}}^-(X)\geq |R-X|$ for all $\emptyset\neq X\subseteq V.$
\end{Theorem}

We can conclude the following statement from Theorem \ref{hyperedmonds1} using the same technique that was used in \cite{book} to conclude Theorem \ref{basic2} from Theorem \ref{basic1}. We therefore omit the proof.

\begin{Theorem}\label{hyperedmonds2}
Let $\mathcal{D}=(V,\mathcal{A})$ be a dypergraph and $R$ a multiset in $V$ of size $k.$ Some $\mathcal{A}'\subseteq \mathcal{A}$ is the dyperedge set of a packing $\{\mathcal{B}_r:r \in R\}$ of spanning $r$-hyperarborescences in $\mathcal{D}$ if and only if the underlying hyperedge set of $\mathcal{A}'$ is the hyperedge set of a packing of $k$ spanning hypertrees and $d_{\mathcal{A}'}^-(v)=k-|R \cap v|$ for all $v \in V.$ 
\end{Theorem}


The above idea of approaching the problem of packing spanning arborescences via matroid intersection is useful for two reasons. On the one hand, one can apply Edmonds' matroid intersection theorem \cite{e3} to get the characterization for the existence of the packing. On the other hand, one can apply Edmonds' weighted matroid intersection algorithm  \cite{e2} to find a packing of minimum total weight. These results of Edmonds are presented in the following theorem. 

\begin{Theorem}\label{matroidintersection}
Let ${\sf M}_1=(S,r_1)$ and ${\sf M}_2=(S,r_2)$ be two matroids on a common ground set $S$ with polynomial independence oracles for ${\sf M}_1$ and ${\sf M}_2$ being available, $\mu$ a positive integer and let $w:S\rightarrow \mathbb{R}$ be a weight function.
\begin{itemize}
	\item [(a)] \cite{e3} A common independent set of size $\mu$  of ${\sf M}_1$ and ${\sf M}_2$ exists  if and only if $r_1(Z)+r_2(S-Z)\geq \mu$ for all $Z \subseteq S$.
	\item [(b)] \cite{e2} We can decide if a common independent set of size $\mu$ of ${\sf M}_1$ and ${\sf M}_2$ exists in polynomial time. Further, if this is the case, then one of minimum weight  can be computed in polynomial time.
\end{itemize}
\end{Theorem}

Giving a similar characterization in terms of matroid intersection for arborescence packings in a more general setting is the main purpose of this article.
\medskip

Firstly, the generalization concerns mixed hypergraphs instead of digraphs. Secondly, the generalization relaxes the condition of the roots of the arborescences being fixed. We are given a mixed hypergraph {\boldmath$\mathcal{H}=(V,\mathcal{A}\cup \mathcal{E})$, a non-negative integer $k$ and functions $f,g$} $:V\rightarrow \mathbb{Z}_{\geq 0},$ 
and we  want to find a packing of $k$ spanning mixed  $r_i$-hyperarborescences $\{\mathcal{B}_i:i\in \{1,\ldots,k\}\}$ such that every vertex $v \in V$ is the root of at least $f(v)$ and at most $g(v)$ of the $k$ hyperarborescences. We call such a packing {\it $(k,f,g)$-flexible}.
The digraphic case of this problem has been successfully treated by Frank in \cite{f}. He gave both a theorem characterizing the digraphs admitting a $(k,f,g)$-flexible packing for given $k,f,g$ and an algorithm to find such a packing if it exists. Recently, this has been generalized to the case of mixed graphs by Gao and Yang in \cite{gy}. 
\medskip

The basic insight of our approach is that, {\sl given a packing of $k$ spanning arborescences, for every vertex $v\in V$, the number of arborescences in the packing whose root is $v$ plus the in-degree of $v$ in the packing is equal to $k.$} We use this fact to show that $(k,f,g)$-flexible packings in mixed hypergraphs can be modeled as the intersection of two matroids. While a hypergraphic analogue of the $k$-forest matroid is maintained as one of the two matroids, the $(k,R)$-partition matroid is replaced by a more general object, a so-called generalized partition matroid. Using Theorem \ref{matroidintersection}(a), this allows to obtain the following characterization for $(k,f,g)$-flexible packings in mixed hypergraphs which is the main contribution of this article. It generalizes the theorem of Gao and Yang. Our proof is completely different from the one in \cite{gy} and works for mixed hypergraphs.

\begin{Theorem}\label{new}
Let $\mathcal{F}=(V,\mathcal{A}\cup \mathcal{E})$ be a mixed hypergraph, $k\in \mathbb{Z}_+$ and $f, g:V\rightarrow \mathbb{Z}_+$ functions.  There exists a $(k,f,g)$-flexible packing of mixed hyperarborescences in $\mathcal{F}$ if and only if   we have
\begin{eqnarray}
	f(v)	& 	\leq 	&	g(v) \ \ \ \ \ \ \ \ \ \ \ \ \ \ \ \ \ \ \ \ \ \ \ \ \text{ for every } v \in V,\label{fg} \\ 
	e_{\mathcal{E}\cup\mathcal{A}}(\mathcal{P})	&	\ge 	&	k(|\mathcal{P}|-1)+f(V-\cup\mathcal{P}) \text{ for every subpartition } \mathcal{P} \text{ of } V,\label{dktgy}\\ 
	e_{\mathcal{E}\cup\mathcal{A}}(\mathcal{P})	&	\ge 	&	k|\mathcal{P}|-g(\cup\mathcal{P}) \ \ \ \ \ \ \ \ \ \ \ \ \text{  for every subpartition } \mathcal{P} \text{ of } V. \label{dktgy2}
\end{eqnarray}
\end{Theorem}

While the technique of matroid intersection for arborescence packings is routinely used as a tool to obtain algorithms for the weighted cases, this is to our best knowledge the first time that a new structural result is obtained via matroid intersection. By Theorem \ref{matroidintersection}(b), the previous observation also yields an algorithm to compute a $(k,f,g)$-flexible packing  of minimum total weight in polynomial time. 
 
\begin{Theorem}\label{algoflex}
Let $\mathcal{F}=(V,\mathcal{A}\cup \mathcal{E})$ be a mixed hypergraph,  $k\in\mathbb{Z}_+, f,g:V\rightarrow \mathbb{Z}_+$  functions and  $w:\mathcal{A}\cup \mathcal{E}\rightarrow \mathbb{R}$  a weight function. Then a $(k,f,g)$-flexible packing of mixed hyperarborescences of minimum weight can be computed in polynomial time, if there exists one.
\end{Theorem}

\section{Definitions}\label{deflab}

For some directed graph (in short, {\it digraph}) $D=(V,A)$ and $r \in V$, an {\it $r$-arborescence} in $D$ is a subgraph of $D$ whose underlying graph is a tree and in which  all the vertices except $r$ have exactly $1$ arc entering. An $r$-arborescence of $D$ is called {\it spanning} if it contains all the vertices of $D.$ By a {\it packing} of arborescences we mean a set of arc-disjoint arborescences. 
A {\it multiset} $R$ is a collection of elements in which an element may appear several times. For some $x \in R$, we denote by {\boldmath$|R \cap x|$} the number of times $x$ is contained in $R$.

A {\it mixed hypergraph} is a tuple {\boldmath$\mathcal{F}$} $=(V,\mathcal{A}\cup \mathcal{E})$ where   {\boldmath $V$} is a  set of vertices, {\boldmath $\mathcal{A}$} is a set of directed hyperedges (dyperedges)  and {\boldmath $\mathcal{E}$} is a set  of hyperedges. A {\it dyperedge} {\boldmath$a$} is a tuple $(tail(a), head(a))$ where {\boldmath $head(a)$} is a single vertex in $V$ and {\boldmath $tail(a)$} is a nonempty subset of $V-head(a)$ and  a {\it hyperedge} is a subset of $V$ of size at least two.

For some $X \subseteq V$, we denote by {\boldmath $\delta_{\mathcal{E}}(X)$} the set of hyperedges $e \in \mathcal{E}$ with $e\cap X, e-X \neq \emptyset$, by {\boldmath $\delta_{\mathcal{A}}^-(X)$} the set of dyperedges $a \in \mathcal{A}$ with $head(a)\in X, tail(a)-X \neq \emptyset$. We use {\boldmath $d_{\mathcal{E}}(X)$} for $|\delta_{\mathcal{E}}(X)|$ and {\boldmath $d_{\mathcal{A}}^-(X)$} for $|\delta_{\mathcal{A}}^-(X)|$. For a single vertex $v$, we use {\boldmath $\delta_{\mathcal{E}}(v)$} instead of $\delta_{\mathcal{E}}(\{v\})$ etc. Given a function $f:V\rightarrow \mathbb{Z}$ and $X \subseteq V$, we use the notation {\boldmath $f(X)$} $=\sum_{v\in X}f(v)$ and hence we consider $f(\emptyset)=0.$

 For some $\mathcal{A}'\subseteq \mathcal{A}$ and $\mathcal{E}'\subseteq \mathcal{E}$, we denote by {\boldmath$V(\mathcal{A}'\cup \mathcal{E}')$} the set of vertices in $V$ which are contained in at least one dyperedge in $\mathcal{A}'$ or hyperedge in $\mathcal{E}'$. A mixed hypergraph without hyperedges is a {\it directed hypergraph (dypergraph)} and a mixed hypergraph without dyperedges is a {\it hypergraph}.  The {\it underlying hypergraph} {\boldmath $\mathcal{H}_{\mathcal{F}}$} $=(V,\mathcal{E}_{\mathcal{A}}\cup \mathcal{E})$ of $\mathcal{F}$ is obtained by replacing every dyperedge $a \in \mathcal{A}$ by the hyperedge $head(a)\cup tail(a)$. For a hyperedge $e \in \mathcal{E}$, its corresponding {\it bundle} {\boldmath $\mathcal{A}_e$} is the set of all possible orientations of $e$, i.e. $\mathcal{A}_e=\{(e-v,v):v \in e\}$. The {\it directed extension} {\boldmath $\mathcal{D}_{\mathcal{F}}$} $=(V,\mathcal{A}\cup \mathcal{A}_{\mathcal{E}})$ of $\mathcal{F}$ is obtained by replacing every hyperedge in $\mathcal{E}$ by its corresponding bundle, i.e. {\boldmath $\mathcal{A}_{\mathcal{E}}$} $=\cup_{e \in \mathcal{E}}\mathcal{A}_e$.
 A packing of $k$ spanning hyperarborescences in $\mathcal{D}_{\mathcal{F}}$ is called {\it $(k,f,g)$-feasible} if every vertex $v \in V$ is the root of at least $f(v)$ and at most $g(v)$ of the $k$ hyperarborescences and for every $e \in \mathcal{E}$, at most one dyperedge of the bundle $\mathcal{A}_e$ is contained in the dyperedge set of the packing.
We say that $\mathcal{F}$ is a {\it mixed graph} if each dyperedge has a tail of size exactly one and each hyperedge contains exactly two vertices. 

{\it Trimming} a dyperedge $a$ means that $a$ is replaced by an arc $uv$ with $v=head(a)$ and $u\in tail(a)$.
{\it Trimming} a hyperedge $e$ means that $e$ is replaced by an arc $uv$ for some $u\neq v\in e$.
The mixed hypergraph $\mathcal{H}$ is called a {\it mixed hyperarborescence} if its dyperedges and hyperedges can be trimmed to get an arborescence.  A {\it mixed $r$-hyperarborescence} for some $r \in V$ is a mixed hyperarborescence together with a vertex $r$  whose dyperedges and hyperedges can be trimmed to get an $r$-arborescence. A hypergraph is called a {\it (spanning) hypertree} if it can be trimmed to a (spanning) arborescence.

Given a packing $\mathcal{B}=\{\mathcal{B}_r:r \in R\}$ of mixed hyperarborescences, we use {\boldmath $\mathcal{A}(\mathcal{B})$} for $\cup_{r \in R}\mathcal{A}(\mathcal{B}_r)$ and {\boldmath $\mathcal{E}(\mathcal{B})$} for $\cup_{r \in R}\mathcal{E}(\mathcal{B}_r)$. Further, given a weight function $w:\mathcal{A}\cup \mathcal{E}\rightarrow \mathbb{R}$, we use {\boldmath $w(\mathcal{B})$} for $w(\mathcal{A}(\mathcal{B}))+w(\mathcal{E}(\mathcal{B}))$.


Given a subpartition $\mathcal{P}$ of $V$, we use {\boldmath $\mathcal{E}(\mathcal{P})$} for $\cup_{X \in \mathcal{P}}\delta_{\mathcal{E}}(X)$ and {\boldmath $\mathcal{A}(\mathcal{P})$} for $\cup_{X \in \mathcal{P}}\delta^-_{\mathcal{A}}(X)$.  We use {\boldmath $e_\mathcal{E}(\mathcal{P}), e_\mathcal{A}(\mathcal{P}), e_{\mathcal{E}\cup\mathcal{A}}(\mathcal{P})$} for the cardinality of $\mathcal{E}(\mathcal{P}),\mathcal{A}(\mathcal{P})$ and $\mathcal{E}(\mathcal{P})\cup \mathcal{A}(\mathcal{P})$, respectively. A hypergraph $\mathcal{H}=(V,\mathcal{E})$ is called {\it partition-connected} if $e_\mathcal{E}(\mathcal{P})\geq |\mathcal{P}|-1$ for every partition $\mathcal{P}$ of $V$.  Further, we use {\boldmath$\cup \mathcal{P}$} for the union of the classes of $\mathcal{P}$.

 Basic notions of matroids which are used in this article can be found in Chapter 5 of \cite{book}.

\section{Relevant matroids}

We now give an overview of the matroids we need for our characterization.

\subsection{Hypergraphic matroids}

While graphic matroids are well-studied, their generalization to hypergraphs has received significantly less attention. The following matroid construction was first observed by Lorea \cite{l}. Given a hypergraph $\mathcal{H}=(V,\mathcal{E})$, let 
{\boldmath$\mathcal{I}_{\mathcal{H}}$} $=\{\mathcal{Z}\subseteq \mathcal{E}:|V(\mathcal{Z}')|>|\mathcal{Z}'| \text{ for all } \emptyset\neq \mathcal{Z}'\subseteq \mathcal{Z}\}.$

\begin{Theorem}
The set $\mathcal{I}_{\mathcal{H}}$  is the set of independent sets of a matroid {\boldmath${\sf M}_{\mathcal{H}}$} on $\mathcal{E}$.
\end{Theorem}

The matroid ${\sf M}_{\mathcal{H}}$ is called the {\it hyperforest matroid} of the hypergraph $\mathcal{H}$.

For our algorithmic result, we need to show that an independence oracle for ${\sf M}_{\mathcal{H}}$ exists. In order to do so, we require the following two preliminaries.
The first result can be found as Corollary 2.6 in \cite{fkk2}.
\begin{Proposition}\label{rangpartcon}
Let $\mathcal{H}=(V,\mathcal{E})$ be a hypergraph. Then $r_{{\sf M}_{\mathcal{H}}}(\mathcal{E})=|V|-1$ if and only if $\mathcal{H}$ is partition-connected.
\end{Proposition}
This result is useful due to the next one which can be found in \cite{book} as a comment after Theorem 9.1.22 stating that the proof of Theorem 9.1.15 is algorithmic.
\begin{Proposition}\label{yfgijl}
There is a polynomial time algorithm that decides whether a given hypergraph $\mathcal{H}$ is partition-connected.
\end{Proposition}
We are now ready to conclude that a polynomial time independence oracle for ${\sf M}_\mathcal{H}$ exists.
\begin{Lemma}\label{qdfvq}
Given a hypergraph $\mathcal{H}=(V,\mathcal{E})$ and $\mathcal{Z}\subseteq \mathcal{E}$, we can decide in polynomial time whether $\mathcal{Z}$ is independent in ${\sf M}_{\mathcal{H}}$. 
\end{Lemma}
\begin{proof}
If $|\mathcal{Z}|\geq |V|$, it follows immediately from the definition of ${\sf M}_\mathcal{H}$ that $\mathcal{Z}$ is dependent in ${\sf M}_\mathcal{H}$. We may hence suppose that $|\mathcal{Z}|\leq |V|-1$. Let a hypergraph $\mathcal{H}'$ be obtained from $(V,\mathcal{Z})$ by adding a set $\mathcal{S}$ of $|V|-1-|\mathcal{Z}|$ hyperedges each of which equals $V$. Observe that $|\mathcal{Z}\cup \mathcal{S}|=|V|-1$.
\begin{Claim}\label{gu}
$\mathcal{Z}$ is independent in ${\sf M}_{\mathcal{H}}$ if and only if $\mathcal{H}'$ is partition-connected.
\end{Claim}
\begin{proof}
First suppose that $\mathcal{Z}$ is independent in ${\sf M}_{\mathcal{H}}$. By definition of ${\sf M}_{\mathcal{H}}$, for any $\mathcal{Z}'\subseteq \mathcal{Z}$, we have $|V(\mathcal{Z}')|>|\mathcal{Z}'|$. For any $\mathcal{Z}'\subseteq \mathcal{Z}\cup \mathcal{S}$ with $\mathcal{Z}'\cap \mathcal{S}\neq \emptyset$, we have $|V(\mathcal{Z}')|=|V|>|\mathcal{Z}\cup \mathcal{S}|\geq|\mathcal{Z}'|$. It follows that $|V(\mathcal{Z}')|>|\mathcal{Z}'|$ for all $\mathcal{Z}'\subseteq \mathcal{Z}\cup \mathcal{S}$ and so by definition, $r_{{\sf M}_{\mathcal{H}'}}(\mathcal{Z}\cup \mathcal{S})=|V|-1$. Now Proposition \ref{rangpartcon} yields that $\mathcal{H}'$ is partition-connected.
\medskip

Now suppose that $\mathcal{H}'$ is partition-connected. It follows from Proposition \ref{rangpartcon} that $r_{{\sf M}_{\mathcal{H}'}}(\mathcal{Z}\cup \mathcal{S})=|V|-1=|\mathcal{Z}\cup \mathcal{S}|$. It follows that ${\sf M}_{\mathcal{H}'}$ is the free matroid, so in particular, $\mathcal{Z}$ is independent in ${\sf M}_{\mathcal{H}'}$. It follows that $\mathcal{Z}$ is also independent in ${\sf M}_{\mathcal{H}}$.
\end{proof}
By Claim \ref{gu} and as $\mathcal{H}'$ can be constructed efficiently, it suffices to check whether $\mathcal{H}'$ is partition-connected. By Proposition \ref{yfgijl}, this can be done in polynomial time.
\end{proof}

We also need the $k$-sum matroid of ${\sf M}_{\mathcal{H}}$, that is the matroid on ground set $\mathcal{E}$ in which a subset of $\mathcal{E}$ is independent if it can be partitioned into $k$  independent sets of ${\sf M}_{\mathcal{H}}$. We call this matroid {\it $k$-hyperforest matroid} and refer to it as {\boldmath${\sf M}_{\mathcal{H}}^k$}. The following formula for the rank function of ${\sf M}_{\mathcal{H}}^k$ was proved by Frank, Kir\'aly and Kriesell \cite{fkk2}.

\begin{Theorem}\label{hyperrank}
For all $\mathcal{Z}\subseteq \mathcal{E}$, we have
$r_{{\sf M}_{\mathcal{H}}^k}(\mathcal{Z})=\min\{e_\mathcal{Z}(\mathcal{P})+k(|V|-|\mathcal{P}|):\mathcal{P}\text{ a partition of }V\}.$
\end{Theorem}

We now extend the previous construction to mixed hypergraphs. Let $\mathcal{F}=(V,\mathcal{A}\cup \mathcal{E})$ be a mixed hypergraph,  $\mathcal{H}_{\mathcal{F}}=(V,\mathcal{E}_{\mathcal{A}}\cup \mathcal{E})$  the underlying hypergraph of $\mathcal{F}$ and  $\mathcal{D}_{\mathcal{F}}=(V,\mathcal{A}\cup \mathcal{A}_{\mathcal{E}})$ the  directed extension  of $\mathcal{F}$. 
We now construct the {\it extended $k$-hyperforest matroid} {\boldmath${\sf M}_{\mathcal{F}}^k$} on $\mathcal{A}\cup\mathcal{A}_{\mathcal{E}}$ from ${\sf M}_{\mathcal{H}_{\mathcal{F}}}^k$ by replacing every $e \in \mathcal{E}$ by $|e|$ parallel copies of itself, associating these elements to the dyperedges in $\mathcal{A}_e$ and associating every element of $\mathcal{E}_\mathcal{A}$ to the corresponding element in $\mathcal{A}$. We give the following formula for the rank function of ${\sf M}_{\mathcal{F}}^k$.

\begin{Proposition}\label{r1}
For all $\mathcal{Z}\subseteq \mathcal{A}\cup \mathcal{A}_{\mathcal{E}}$, we have
\begin{equation}
r_{{\sf M}_{\mathcal{F}}^k}(\mathcal{Z})=\min\{|\mathcal{Z}\cap \mathcal{A}(\mathcal{P})|+|\{e \in \mathcal{E}(\mathcal{P}):\mathcal{Z}\cap \mathcal{A}_e \neq \emptyset\}|+k(|V|-|\mathcal{P}|):\mathcal{P}\text{ a partition of }V\}.
\end{equation}
\end{Proposition}

\begin{proof}
Let {\boldmath$\mathcal{Z}'$} be obtained from $\mathcal{Z}$ by deleting all but one element of $\mathcal{Z} \cap \mathcal{A}_e$ for all $e \in \mathcal{E}$ with $|\mathcal{Z} \cap \mathcal{A}_e|\geq 2$. As all elements in $\mathcal{A}_e$ are parallel in ${\sf M}_{\mathcal{F}}^k$, we obtain that $r_{{\sf M}_{\mathcal{F}}^k}(\mathcal{Z})=r_{{\sf M}_{\mathcal{F}}^k}(\mathcal{Z}')$. As $|\mathcal{Z}' \cap \mathcal{A}_e|\leq 1$ for every $e \in \mathcal{E}$, there exists a matroid ${\sf M}'$ that is isomorphic to ${\sf M}^k_{\mathcal{H}}$, is a restriction of ${\sf M}^k_{\mathcal{F}}$ and whose ground set contains $\mathcal{Z}'$. It follows from Theorem \ref{hyperrank} that
$r_{{\sf M}_{\mathcal{F}}^k}(\mathcal{Z})=r_{{\sf M}_{\mathcal{F}}^k}(\mathcal{Z}')=r_{{\sf M}'}(\mathcal{Z}')
=\min\{|\mathcal{Z}'\cap\mathcal{A}(\mathcal{P})|+|\mathcal{Z}'\cap\mathcal{E}(\mathcal{P})|+k(|V|-|\mathcal{P}|):\mathcal{P}\text{ a partition of }V \}=\min\{|\mathcal{Z}\cap \mathcal{A}(\mathcal{P})|+|\{e \in \mathcal{E}(\mathcal{P}):\mathcal{Z}\cap \mathcal{A}_e \neq \emptyset\}|+k(|V|-|\mathcal{P}|):\mathcal{P}\text{ a partition of }V\}.$
\end{proof}

Again, for the algorithmic part, we need to show that an independence oracle for ${\sf M}_{\mathcal{F}}^k$ is available. We need the following preliminary result on matroids which was proven by Edmonds \cite{e1}.
\begin{Proposition}\label{tfuyg}
Let ${\sf M}$ be a matroid such that a polynomial time independence oracle for ${\sf M}$ is available and $k$ a positive integer. Then a polynomial time independence oracle for the $k$-sum of ${\sf M}$ is also available.
\end{Proposition}
\begin{Lemma}\label{obstacle}
Given a mixed hypergraph $\mathcal{F}=(V,\mathcal{A}\cup \mathcal{E})$ and $\mathcal{Z}\subseteq \mathcal{A}\cup \mathcal{A}_\mathcal{E}$, we can decide in polynomial time if $\mathcal{Z}$ is independent in ${\sf M}_{\mathcal{F}}^k$.
\end{Lemma}
\begin{proof}
If $\mathcal{Z}$ contains at least 2 elements of $\mathcal{A}_e$ for some $e \in \mathcal{E}$, then $\mathcal{Z}$ is dependent in ${\sf M}_{\mathcal{F}}^k$ by definition. Otherwise, there is a matroid ${\sf M}'$ that is isomorphic to ${\sf M}^k_{\mathcal{H}_{\mathcal{F}}}$, is a restriction of ${\sf M}^k_{\mathcal{F}}$ and whose ground set contains $\mathcal{Z}$. Further, ${\sf M}'$ can be found efficiently. It therefore suffices to prove that a polynomial time independence oracle for ${\sf M}^k_{\mathcal{H}_{\mathcal{F}}}$ is available. This follows immediately from Lemma \ref{qdfvq} and Proposition \ref{tfuyg}.
\end{proof}
\subsection{Generalized partition matroids}

The other matroid we consider is called a generalized partition matroid and plays the role of the $(k,R)$-partition matroid. It has been considered in a slightly weaker form in \cite{book}, see Problem 5.3.4.
\medskip

 Let $\{${\boldmath$S_1$}$,\dots,${\boldmath$S_n$}$\}$ be a partition of a  set {\boldmath$S$} and {\boldmath$\mu$}, {\boldmath$\alpha_i$}, {\boldmath$\beta_i$} $\in \mathbb Z$ for all $i\in\{1,\dots,n\}.$  For $Z\subseteq S$ and $i\in\{1,\dots,n\},$ let {\boldmath$z_i$} $=|Z\cap S_i|.$ Let 
\begin{itemize}[label={}]
\item {\boldmath$\mathcal{I}$} $=\{Z\subseteq S:z_i\leq \beta_i \text{ for all }i\in\{1,\ldots,n\},\sum_{i=1}^n \max\{\alpha_i, z_i\}\leq \mu\}$,
 \item{\boldmath$\mathcal{B}$} $=\{Z \subseteq S: \alpha_i\leq z_i\leq \beta_i \text{ for all }i\in\{1,\ldots,n\},|Z|=\mu\}$
    and 
 \item {\boldmath$r(Z)$} $=\min\{\sum_{i=1}^n \min \{\beta_i,z_i\},\mu-\sum_{i=1}^n \max\{\alpha_i-z_i,0\}\}\text{ for all } Z\subseteq S.$
\end{itemize}
\begin{Theorem}\label{genpart}
There exists  a matroid ${\sf M}$ whose  set of independent sets is $\mathcal{I}$, set of bases is $\mathcal{B}$ and  rank function is $r$ if and only if 

\begin{eqnarray}
	 &&\hskip .6truecm\max\{\alpha_i, 0\}\hskip .64truecm \leq \hskip .6truecm\min\{\beta_i, |S_i|\} \ \text{ for all } i\in\{1,\ldots,n\},\label{genpartcond1}\\
	&&\sum_{i=1}^n \max\{\alpha_i,0\} \leq \mu \leq \sum_{i=1}^n \min\{\beta_i,|S_i|\}. \label{genpartcond2}
\end{eqnarray}
\end{Theorem}

The matroid ${\sf M}$ in Theorem \ref{genpart} is called {\it generalized partition matroid}.

\begin{proof}
First suppose that ${\sf M}$ is a matroid and let $Z \in \mathcal{B}$. For all $i \in \{1,\ldots,n\}$, this yields $\max\{\alpha_i,0\}\leq z_i\leq \min\{\beta_i,|S_i|\}$. Further, we obtain $\sum_{i=1}^n\max\{\alpha_i,0\}\leq \sum_{i=1}^n z_i=\mu$ and $\sum_{i=1}^n\max\{\beta_i,|S_i|\}\geq \sum_{i=1}^n z_i=\mu$.

We now show sufficiency through three claims.

\begin{Claim}
$\mathcal{I}$ forms the collection of independent sets of a matroid ${\sf M}$, i.e. $\mathcal{I}$ satisfies the following 3 independence axioms:
\begin{itemize}[label={}]
\item[($I_0$)]  $\emptyset \in \mathcal{I}$,
\item[($I_1$)] if $Z \subset Z'$ and $Z' \in \mathcal{I}$, then $Z \in \mathcal{I}$,
\item[($I_2$)] if $Z, Z' \in \mathcal{I}$ and $|Z|<|Z'|$, then there exists some $x \in Z'-Z$ such that $Z \cup \{x\} \in \mathcal{I}$.
\end{itemize}
\end{Claim}

\begin{proof}
$(I_0)$: By \eqref{genpartcond1}, we have $\beta_i\geq 0=|\emptyset \cap S_i|$ for $i \in \{1,\ldots,n\}$. Further, \eqref{genpartcond2} yields $\sum_{i=1}^n \max\{\alpha_i,|\emptyset \cap S_i|\}= \sum_{i=1}^n \max\{\alpha_i,0\}\leq \mu$. This yields $\emptyset \in \mathcal{I}$.

$(I_1)$: Let $Z\subset Z' \in \mathcal{I}$. Then   $z_i\le z_i'\le\beta_i$ for $i=1,\ldots,n$ and $\sum_{i=1}^n \max\{\alpha_i, z_i\}\leq\sum_{i=1}^n \max\{\alpha_i, z'_i\}\leq \mu,$ so we have $Z\in \mathcal{I}$.

$(I_2)$: Let $Z,Z' \in \mathcal{I}$ and $|Z|<|Z'|.$ Let {\boldmath$J$}$=\{j\in \{1,\ldots,n\}:z_j<z_j'\}$. Observe that $J \neq \emptyset$ as $|Z|<|Z'|$. For all $j \in J$, let $x_j \in (S_j \cap Z')-Z$ and $Z^j=Z \cup \{x_j\}$. Observe that for all $j \in J$, we have $z_j^j\leq z'_j\leq\beta_j$ and $z^j_i\leq z_i\leq \beta_i$ for all $i \in \{1,\ldots,n\}-\{j\}$. In order to prove that $Z^j \in \mathcal{I}$ for some $j \in J$, it remains to show that there is some $j \in J$ with $\sum_{i=1}^n \max\{\alpha_i, z^j_i\}\leq \mu$. If there is some $j \in J$ with $z_j<\alpha_j$, then $\sum_{i=1}^n \max\{\alpha_i, z^j_i\}=\sum_{i=1}^n \max\{\alpha_i, z_i\}\leq\mu$, so we are done. We may hence suppose that $z_j\geq \alpha_j$ for all $j \in J$. This yields $\max\{\alpha_j-z'_j,0\}\geq 0=\max\{\alpha_j-z_j,0\}$ for all $j \in J$. For all $i \in \{1,\ldots,n\}-J$, we have $z_i\geq z_i'$ yielding $\max\{\alpha_i-z'_i,0\}\geq\max\{\alpha_i-z_i,0\}$. This yields $\max\{\alpha_i-z'_i,0\}\geq\max\{\alpha_i-z_i,0\}$ for all $i \in \{1,\ldots,n\}$. As $|Z|<|Z'|$ and $Z'\in \mathcal{I}$,  for some arbitrary $j \in J$ we obtain

\begin{align*}
\sum_{i=1}^n\max\{\alpha_i,z^j_i\}&\leq \sum_{i=1}^n\max\{\alpha_i,z_i\}+1\\
&= \sum_{i=1}^n\max\{\alpha_i-z_i,0\}+\sum_{i=1}^nz_i+1\\
&\leq \sum_{i=1}^n\max\{\alpha_i-z'_i,0\}+\sum_{i=1}^nz'_i\\
&=\sum_{i=1}^n\max\{\alpha_i,z'_i\}\\
&\leq \mu.
\end{align*}

\end{proof}

\begin{Claim}\label{ind}
$Z \in \mathcal{B}$ if and only if $Z$ is a maximal element in $\mathcal{I}$.
\end{Claim}

\begin{proof}
First let $Z \in \mathcal{B}$. We obtain $\mu=|Z|=\sum_{i =1}^nz_i=\sum_{i =1}^n\max\{\alpha_i,z_i\}$. As $\beta_i \geq z_i$ for all $i \in \{1,\ldots,n\}$, we have $Z \in \mathcal{I}$. Further, for any proper superset $Z'$ of $Z$, we have $\sum_{i =1}^n\max\{\alpha_i,z'_i\}\geq |Z'|>|Z|=\mu$, so $Z' \notin \mathcal{I}$. It follows that $Z$ is maximally in $\mathcal{I}$.

Now let $Z$ be a maximal element in $\mathcal{I}$. If $z_j<\alpha_j$ for some $j \in \{1,\ldots,n\}$, let $x \in S_j-Z$ and let $Z'=Z \cup \{x\}$. As $\alpha_j\leq \beta_j$ and $Z \in \mathcal{I}$, we obtain $z_j'\leq z_j+1\leq \alpha_j\leq \beta_j$ and $z_i'=z_i\leq \beta_i$ for all  $i \in \{1,\ldots,n\}-\{j\}$. Further, $\sum_{i=1}^n\max\{\alpha_i,z'_i\}=\sum_{i=1}^n\max\{\alpha_i,z_i\}\leq \mu$, so $Z' \in \mathcal{I}$. This contradicts the maximality of $Z$.  We obtain that $z_j\geq \alpha_j$ for all $j \in \{1,\ldots,n\}$.

If $|Z|<\mu$, by $\sum_{i=1}^n\min\{\beta_i,|S_i|\}\geq \mu$, there exists some $j \in \{1,\ldots,n\}$ such that $z_j<\min\{\beta_j,|S_j|\}$. Let $x \in S_j-Z$ and let $Z'=Z \cup \{x\}$. We have $z_j'\leq z_j+1\leq \beta_j$ and $z_i'=z_i\leq \beta_i$ for all  $i \in \{1,\ldots,n\}-\{j\}$. Further, $\sum_{i=1}^n\max\{\alpha_i,z'_i\}=\sum_{i=1}^n\max\{\alpha_i,z_i\}+1=|Z|+1\leq \mu$, so $Z' \in \mathcal{I}$. This contradicts the maximality of $Z$. It follows that $Z \in \mathcal{B}$.

\end{proof}

\begin{Claim} 
The rank function of ${\sf M}$ is $r.$
\end{Claim}

\begin{proof}
 Let {\boldmath$Z$} $\subseteq S$ and {\boldmath$Y$} be a maximal element of $\mathcal{I}$ in $Z.$ 

As $Y \in \mathcal{I}$ and $Y \subseteq Z$, we obtain $y_i\leq \min\{\beta_i,z_i\}$ for all $i \in\{1,\ldots,n\}$. This yields $r_{{\sf M}}(Z)=|Y|\leq \sum_{i=1}^n\min \{\beta_i,z_i\}$.
Further, as $y_i \leq z_i$ for all $i \in\{1,\ldots,n\}$ and $Y \in \mathcal{I}$, we obtain $\sum_{i=1}^ny_i+\sum_{i=1}^n\max\{\alpha_i-z_i,0\}=\sum_{i=1}^n\max\{\alpha_i-z_i+y_i,y_i\}\leq \sum_{i=1}^n\max\{\alpha_i,y_i\}\leq \mu$, so $r_{{\sf M}}(Z)=|Y|\leq \mu-\sum_{i=1}^n\max\{\alpha_i-z_i,0\}$. It follows that $r_{{\sf M}}(Z)\leq \min\{ \sum_{i=1}^n\min \{\beta_i,z_i\},  \mu-\sum_{i=1}^n\max\{\alpha_i-z_i,0\}\}=r(Z)$.
\medskip

Let {\boldmath $J$}$=\{j \in \{1,\ldots,n\}:y_j<\min\{z_j,\beta_j\}\}$. If $J =\emptyset$, we obtain $r_{\sf M}(Z)=|Y|\geq \sum_{i=1}^n\min\{z_i,\beta_i\}\geq \min\{ \sum_{i=1}^n\min \{\beta_i,z_i\},  \mu-\sum_{i=1}^n\max\{\alpha_i-z_i,0\}\}$, so we are done. We may hence suppose that $J \neq \emptyset$.
\medskip

For all $j \in J$, let $x_j \in (S_j\cap Z)-Y$ and let $Y^j=Y \cup \{x_j\}$. Observe that for all $j \in J$, we have $y^j_j=y_j+1\leq \min\{z_j,\beta_j\}$ and $y^j_i=y_i\leq \min\{z_i,\beta_i\}$ for all $i \in \{1,\ldots,n\}-\{j\}$. If $\sum_{i=1}^n\max\{\alpha_i,y_i\}<\mu$, then for some arbitrary $j \in J$, we have $\sum_{i=1}^n\max\{\alpha_i,y^j_i\}\leq \sum_{i=1}^n\max\{\alpha_i,y_i\}+1\leq\mu$, so $Y^j \in\mathcal{I}$, a contradiction to the maximality of $Y$. This yields $\sum_{i=1}^n\max\{\alpha_i,y_i\}=\mu$.
\medskip

If there is some $j \in J$ such that $y_j<\alpha_j$, then $\sum_{i=1}^n\max\{\alpha_i,y^j_i\}=\sum_{i=1}^n\max\{\alpha_i,y_i\}=\mu$, so $Y^j \in\mathcal{I}$, a contradiction to the maximality of $Y$. This yields $y_j\geq \alpha_j$ for all $j \in J$, so, as $y_j \leq z_j$, we obtain $\max\{\alpha_j-y_j,0\}=0=\max\{\alpha_j-z_j,0\}$. For all $i \in\{1,\ldots,n\}-J$, we have either $y_i=\beta_i$ or $y_i=z_i$. If $y_i=\beta_i$, by $y_i\leq z_i$ and $\alpha_i\leq \beta_i$, we obtain $\max\{\alpha_i-y_i,0\}=0=\max\{\alpha_i-z_i,0\}$. If $y_i=z_i$, we clearly obtain $\max\{\alpha_i-y_i,0\}=\max\{\alpha_i-z_i,0\}$. It follows that $\max\{\alpha_i-y_i,0\}=\max\{\alpha_i-z_i,0\}$ holds for all $i \in \{1,\ldots,n\}$.
\medskip

This yields $r_{{\sf M}}(Z)=\sum_{i=1}^ny_i=\sum_{i=1}^n\max \{\alpha_i,y_i\}-\sum_{i=1}^n\max \{\alpha_i-y_i,0\}=\mu -\sum_{i=1}^n\max \{\alpha_i-z_i,0\}\geq \min\{ \sum_{i=1}^n\min \{\beta_i,z_i\},  \mu-\sum_{i=1}^n\max\{\alpha_i-z_i,0\}\}=r(Z)$.

\end{proof}
The three previous claims yield the theorem.
\end{proof}

\begin{Corollary}\label{gpmd}
Let $\mathcal{D}=(V,\mathcal{A}')$ be a directed hypergraph and $f,g:V \rightarrow \mathbb{Z}_{+}$ integer functions such that the following two conditions are satisfied:
\begin{eqnarray}
	 \max\{k-g(v), 0\}			&	\leq 	&	\min\{k-f(v), d_{\mathcal{A}'}^-(v)\} \ \text{ for all } v \in V,\label{gpc1}\\
	\sum_{v\in V} \max\{k-g(v),0\} 	&	\leq	&	 k(|V|-1) \leq \sum_{v\in V}  \min\{k-f(v),d_{\mathcal{A}'}^-(v)\}. \label{gpc2}
\end{eqnarray}
Then  $\{\mathcal{Z}\subseteq \mathcal{A}':|\mathcal{Z}|=k(|V|-1), f(v)\leq k-d_{\mathcal{Z}}^-(v)\leq g(v)$ for all $v\in V\}$ is the set of bases of a matroid {\boldmath${\sf M}_{\mathcal{D}}^{(k,f,g)}$} on $\mathcal{A}'$ with rank function
\begin{equation*}
r_{{\sf M}_\mathcal{F}^{(k,f,g)}}(\mathcal{Z})=\min\{\sum_{v \in V}\min\{k-f(v),d_{\mathcal{Z}}^-(v)\},k(|V|-1)-\sum_{v \in V}\max\{k-g(v)-d_{\mathcal{Z}}^-(v),0\}\}.
\end{equation*}
\end{Corollary}

\begin{proof}
Let {\boldmath$S$} $={\mathcal{A}'},$ {\boldmath$\mu$} $=k(|V|-1)$ and {\boldmath$S_v$} $=\delta_{\mathcal{A}'}^-(v),$ {\boldmath$\alpha_v$} $=k-g(v)$ and {\boldmath$\beta_v$} $=k-f(v)$ for all $v \in V$. Then \eqref{gpc1} and \eqref{gpc2} coincide with \eqref{genpartcond1} and \eqref{genpartcond2}, respectively. We can therefore apply Theorem \ref{genpart} from which the statement follows immediately.
\end{proof}

The following is an immediate corollary of the definition of ${\sf M}_{\mathcal{D}}^{(k,f,g)}$.
\begin{Lemma}\label{oracle}
Given a directed hypergraph $\mathcal{D}=(V,\mathcal{A}')$, $f,g:V \rightarrow \mathbb{Z}_{+}$ integer functions satisfying \eqref{gpc1} and \eqref{gpc2} and $\mathcal{Z}\subseteq \mathcal{A}'$, we can decide in polynomial time if $\mathcal{Z}$ is independent in ${\sf M}_{\mathcal{D}}^{(k,f,g)}$.
\end{Lemma}

\section{Flexible packings by matroid intersection}

We are now ready to show how to model packings of spanning mixed hyperarborescences with flexible roots via matroid intersection. 
\medskip

The following observation allows us to reduce the problem of 
flexible packings in a mixed hypergraph to finding 
feasible packings in its directed extension.

\begin{Lemma}\label{flexfeas}
Let $\mathcal{F}=(V,\mathcal{A}\cup \mathcal{E})$ be a mixed hypergraph, $\mathcal{D}_{\mathcal{F}}=(V,\mathcal{A}\cup \mathcal{A}_\mathcal{E})$ its directed extension, $f,g:V\rightarrow \mathbb{Z}_+$ integer valued functions and $k \in \mathbb{Z}_+$.
Then $\mathcal{F}$ has a $(k,f,g)$-flexible packing if and only if  $\mathcal{D}_{\mathcal{F}}$ has a $(k,f,g)$-feasible packing.
\end{Lemma}

\begin{proof}
First suppose that $\mathcal{F}$ has a $(k,f,g)$-flexible packing {\boldmath$\mathcal{B}$} $=\{${\boldmath$\mathcal{B}_i$} $:i \in\{1,\ldots,k\}\}$. Then there is a $(k,f,g)$-flexible packing {\boldmath$B$} $=\{${\boldmath$B_i$} $:i \in \{1,\ldots,k\}\}$ of arborescences such that $B_i$ is a trimming of $\mathcal{B}_i$ for all $i \in \{1,\ldots,k\}$. For every $e \in \mathcal{E}$, let $v_e$ be the head of the arc to which $e$ is trimmed in $B$ and let $\vec{e}\in \mathcal{A}_e$ be the orientation of $e$ where $v_e$ is chosen to be its head. Then the set of dyperedges $\mathcal{A}(\mathcal{B})\cup \{\vec{e}:e \in \mathcal{E}(\mathcal{B})\}$ can be trimmed to $B$ and it contains at most one dyperedge of $\mathcal{A}_e$ for all $e \in \mathcal{E}$. It follows that $\mathcal{D}_{\mathcal{F}}$ has a $(k,f,g)$-feasible packing.

Now suppose that $\mathcal{D}_{\mathcal{F}}$ has a $(k,f,g)$-feasible packing $\mathcal{B}$. By definition, $\mathcal{A}_{\mathcal{E}}\cap \mathcal{A}(\mathcal{B})$ contains at most one dyperedge in $\mathcal{A}_e$ for all $e \in \mathcal{E}$. Replacing a dyperedge in $\mathcal{A}_e\cap \mathcal{A}(\mathcal{B})$ by $e$ for all $e \in \mathcal{E}$, we obtain the dyper- and hyperedge set of a $(k,f,g)$-flexible packing in $\mathcal{F}$.
\end{proof}

%


We now give the characterization via matroid intersection which is the core of this article.


\begin{Theorem}\label{matroidflex}
Let $\mathcal{F}=(V,\mathcal{A}\cup \mathcal{E})$ be a mixed hypergraph, $\mathcal{D}_{\mathcal{F}}=(V,\mathcal{A}'=\mathcal{A}\cup \mathcal{A}_\mathcal{E})$ its directed extension, $f,g:V\rightarrow \mathbb{Z}_+$ integer valued functions and $k \in \mathbb{Z}_+$.
Suppose that \eqref{gpc1} and \eqref{gpc2} are satisfied.
Then the dyperedge sets of the $(k,f,g)$-feasible packings in $\mathcal{D}_{\mathcal{F}}$ are exactly the common independent sets of size $k(|V|-1)$ of ${\sf M}^k_{\mathcal{F}}$ and ${\sf M}_{\mathcal{D}_{\mathcal{F}}}^{(k,f,g)}$.
\end{Theorem}

\begin{proof} 
For the sake of simplicity, we use {\boldmath${\sf M}_1$ and ${\sf M}_2$} for ${\sf M}^k_{\mathcal{F}}$ and ${\sf M}_{\mathcal{D}_{\mathcal{F}}}^{(k,f,g)}$, respectively. 
As \eqref{gpc1} and \eqref{gpc2} are satisfied, Corollary \ref{gpmd} yields that ${\sf M}_2$ is well-defined.
First, let {\boldmath$\mathcal{Z}$} be the dyperedge set of a $(k,f,g)$-feasible packing in $\mathcal{D}_{\mathcal{F}}$. Then the underlying undirected hypergraph of $(V,\mathcal{Z})$ is the union of $k$ hyperedge-disjoint spanning hypertrees and  $\mathcal{Z}$ contains at most one dyperedge of the bundle $A_e$ for all $e\in \mathcal{E}$. It follows that $\mathcal{Z}$ is an independent set of ${\sf M}_1$. As $\mathcal{Z}$ is the dyperedge set of a packing of $k$ spanning hyperarborescences, we have $|\mathcal{Z}|=k(|V|-1)$. Since the packing is $(k,f,g)$-feasible, every vertex $v$ is the root of at least $f(v)$ and at most $g(v)$ of the $k$ spanning  hyperarborescences of the packing. It follows that $k-g(v)\leq d_\mathcal{Z}^-(v)\leq k-f(v)$, so $\mathcal{Z}$ is an independent set of ${\sf M}_2$.
\medskip

Now let {\boldmath$\mathcal{Z}$} be a common independent set of ${\sf M}_1$ and ${\sf M}_2$ of size $k(|V|-1)$. Then, by $f \geq 0$, the underlying hypergraph of $(V,\mathcal{Z})$ is the union of $k$ hyperedge-disjoint spanning hypertrees and  $d_\mathcal{Z}^-(v)\leq k-f(v)\leq k$ for all $v \in V$. Let $R$ be the multiset in $V$ in which every vertex $v$ in $V$ is contained $k-d^-_{\mathcal{Z}}(v)$ times. Observe that this value is nonnegative for all $v \in V$. As $|\mathcal{Z}|=k(|V|-1)$, we have $|R|=\sum_{v \in V}(k-d^-_{\mathcal{Z}}(v))=k|V|-|\mathcal{Z}|=k|V|-k(|V|-1)=k$. Also, by construction $d^-_{\mathcal{Z}}(v)=k-|R \cap v|$ for all $v \in V$. Theorem \ref{hyperedmonds2} therefore implies that $\mathcal{Z}$ is the set of dyperedges of the union of $k$ dyperedge-disjoint spanning hyperarborescences with root set $R$, so each $v\in V$ is the root of $k-d_\mathcal{Z}^-(v)$ of them. As $\mathcal{Z}$ is an independent set of ${\sf M}_2$, we obtain that $k-g(v)\leq d_\mathcal{Z}^-(v)\leq k-f(v)$, so $f(v)\leq k-d_\mathcal{Z}^-(v)\leq g(v)$. Finally, as $\mathcal{Z}$ is independent in ${\sf M}_1$, $\mathcal{Z}$ contains at most one dyperedge of the bundle $A_e$ for all $e\in E$. It follows that $\mathcal{Z}$ is the dyperedge set of a $(k,f,g)$-feasible packing in $\mathcal{D}_{\mathcal{F}}$.
\end{proof}

\section{The proofs of Theorems \ref{new} and \ref{algoflex}}

This section is dedicated to concluding Theorems \ref{new} and \ref{algoflex} from the matroidal characterization found in Theorem \ref{matroidflex}. We first need two slightly technical lemmas. The first one shows the necessity in Theorem \ref{new} and is also used in the proof of Theorem \ref{algoflex}.
\begin{Lemma}\label{necessity}
Let $\mathcal{F}=(V,\mathcal{A}\cup \mathcal{E})$ be a mixed hypergraph, $f,g:V \rightarrow \mathbb{Z}_{+}$ integer functions and $k \in \mathbb{Z}_{+}$. If there exists a $(k,f,g)$-flexible packing  in $\mathcal{F},$ then  \eqref{fg}, \eqref{dktgy} and \eqref{dktgy2} are satisfied.
\end{Lemma}

\begin{proof} By Lemma \ref{flexfeas}, $\mathcal{D}_{\mathcal{F}}$ has a $(k,f,g)$-feasible packing {\boldmath$\mathcal{B}$} $=\{${\boldmath$\mathcal{B}_i$} $:i \in \{1,\ldots,k\}\}$. Let {\boldmath$s(X)$} denote the number of hyperarborescences in $\mathcal{B}$ whose roots are in the vertex set $X.$ Then for every vertex $v \in V,$ $f(v)\le s(v)\le g(v)$ and hence \eqref{fg} is satisfied. Let now  {\boldmath$\mathcal{P}$} be a subpartition of $V.$ Let {\boldmath$\mathcal{Z}$} be the dyperedge set of $\mathcal{B}$. By definition, $\mathcal{Z}$ contains at most one dyperedge in $\mathcal{A}_e$ for all $e \in \mathcal{E}$. It follows that $e_{\mathcal{E}\cup\mathcal{A}}(\mathcal{P})\ge \sum_{X\in \mathcal{P}}d_{\mathcal{Z}}^-(X).$ Since $\mathcal{B}$ is a packing of spanning hyperarborescences, we have $d_{\mathcal{Z}}^-(X)\ge k-s(X).$ Thus $e_{\mathcal{E}\cup\mathcal{A}}(\mathcal{P})\ge k|\mathcal{P}|-s(\cup\mathcal{P}).$ Since $s(V)=k$ and $f(v)\le s(v)\le g(v)$ for all $v \in V,$ we have $s(\cup\mathcal{P})=s(V)-s(V-\cup\mathcal{P})\leq k-f(V-\cup\mathcal{P})$, yielding \eqref{dktgy}. Further, we have $s(\cup\mathcal{P})\leq g(\cup\mathcal{P})$ yielding \eqref{dktgy2}.
\end{proof}

The second one allows to use Theorem \ref{matroidflex} in the proofs of Theorems \ref{new} and \ref{algoflex}.
\begin{Lemma}\label{condgenpart}
Let $\mathcal{F}=(V,\mathcal{A}\cup \mathcal{E})$ be a mixed hypergraph, $\mathcal{D}_{\mathcal{F}}=(V,\mathcal{A}'=\mathcal{A}\cup \mathcal{A}_{\mathcal{E}})$ its directed extension, $f,g:V \rightarrow \mathbb{Z}_{+}$ integer functions and $k \in \mathbb{Z}_{+}$.
If \eqref{fg}, \eqref{dktgy} and \eqref{dktgy2} are satisfied, then \eqref{gpc1} and \eqref{gpc2} are satisfied.
\end{Lemma}

\begin{proof} 
By \eqref{dktgy} for $\mathcal{P}=\emptyset$ and $f\ge 0$, we obtain $k=-k(|\emptyset|-1)\geq f(V-\emptyset)-e_{\mathcal{E}\cup\mathcal{A}}(\emptyset)=f(V).$
By \eqref{dktgy2} for $\mathcal{P}=\{V\}$, we obtain $k=k|\{V\}|\leq g(V)+e_{\mathcal{E}\cup\mathcal{A}}(\{V\})=g(V).$ We obtain $f(V)\leq k \leq g(V)$.

By \eqref{fg}, $k-g(v)\le k-f(v).$ By $f(V)\leq k$ and $f\geq 0$, we obtain $0\le k-f(V)\le k-f(v).$ By \eqref{dktgy2} for $\mathcal{P}=\{v\}$, we obtain $k-g(v)=k|\mathcal{P}|-g(\cup \mathcal{P})\le e_{\mathcal{E}\cup\mathcal{A}}(\mathcal{P})=e_{\mathcal{E}\cup\mathcal{A}}(\{v\})=d_{\mathcal{A}'}^-(v).$ Then, by $0\le d_{\mathcal{A}'}^-(v),$  \eqref{gpc1} follows.
\medskip

Let {\boldmath$V'$} $=\{v\in V: k-g(v)>0\}.$ 	 
We obtain $\sum_{v\in V}  \max\{k-g(v),0\}=\sum_{v\in V'}(k-g(v))=k(|V|-1)-(g(V')+k(|V-V'|-1)).$ If $V' \neq V$, by $g \geq 0$ we obtain $g(V')+k(|V-V'|-1)\geq g(V')\geq 0$. If $V'=V$, by $g(V)\geq k$, we obtain $g(V')+k(|V-V'|-1)= g(V)-k\geq 0$. This yields the first inequality of \eqref{gpc2}.
\medskip

Let {\boldmath$V''$} $=\{v\in V: k-f(v)<d_{\mathcal{A}'}^-(v)\}$. Let $\mathcal{P}=\{\{v\}:v\in V-V''\}$. Note that any element in ${\cal E}(\mathcal{P})\cup {\cal A}({\cal P})$ provides at least one dyperedge in ${\cal A}'$ entering a vertex in $V-V''.$ 
Hence, by \eqref{dktgy}, we obtain
\begin{align*}
 \sum_{v\in V} \min\{k-f(v),d_{\mathcal{A}'}^-(v)\}&=\sum_{v\in V''}(k-f(v))+\sum_{v\in V-V''}d_{\mathcal{A}'}^-(v)\\
&\ge (k|V''|-f(V''))+e_{\mathcal{E}\cup\mathcal{A}}(\mathcal{P})\\
&\geq (k|V''|-f(V''))+k(|\mathcal{P}|-1)+f(V-\cup\mathcal{P})\\
&= (k|V''|-f(V''))+k(|V-V''|-1)+f(V'')\\
&=k(|V|-1).
\end{align*}
 It follows that the second inequality of \eqref{gpc2} is  satisfied.
\end{proof}



We now show that Theorems \ref{matroidflex} and \ref{matroidintersection}(a) imply Theorem \ref{new}.

\begin{proofof}{Theorem \ref{new}}
Necessity is proved in Lemma \ref{necessity}. 

To see  sufficiency, suppose that  \eqref{fg}, \eqref{dktgy} and \eqref{dktgy2} are satisfied. 
 Let $\mathcal{D}_{\mathcal{F}}=(V,\mathcal{A'}=\mathcal{A}\cup \mathcal{A}_{\mathcal{E}})$ be the directed extension of $\mathcal{F}$. By Lemma \ref{condgenpart},  \eqref{gpc1} and \eqref{gpc2} are satisfied. It now follows from Corollary \ref{gpmd} that ${\sf M}_{\mathcal{D}_{\mathcal{F}}}^{(k,f,g)}$ is well-defined. Again, we use {\boldmath${\sf M}_1$ and ${\sf M}_2$} for ${\sf M}^k_{\mathcal{F}}$ and ${\sf M}_{\mathcal{D}_{\mathcal{F}}}^{(k,f,g)}$, respectively. 
Suppose for a contradiction that no $(k,f,g)$-flexible packing of mixed hyperarborescences exists in $\mathcal{F}$.
By Lemma \ref{flexfeas}, no $(k,f,g)$-feasible packing of hyperarborescences exists in $\mathcal{D}_{\mathcal{F}}$. Now Theorem \ref{matroidflex} implies that ${\sf M}_1$ and ${\sf M}_2$ do not have a common independent set of size $k(|V|-1)$.
The next result allows to fix a certain structure leading to a contradiction later.
\begin{Claim}\label{part1}
There exist a partition $\mathcal{P}$ of $V$ and  $\mathcal{K} \subseteq \mathcal{E}(\mathcal{P})$ such that
\begin{equation}\label{pkz}
k(|\mathcal{P}|-1)>|\mathcal{K}|+r_{{\sf M}_2}(\mathcal{\mathcal{A}(\mathcal{P})\cup \mathcal{A}_{\mathcal{E}(\mathcal{P})-\mathcal{K}}}).
\end{equation}
\end{Claim}

\begin{proof}
As ${\sf M}_1$ and ${\sf M}_2$ do not have a common independent set of size $k(|V|-1)$, Theorem \ref{matroidintersection}(a) implies that there exists a dyperedge set {\boldmath$\mathcal{Z}'$} $\subseteq \mathcal{A}'$ such that  $k(|V|-1)>r_{{\sf M}_1}(\mathcal{Z}')+r_{{\sf M}_2}(\mathcal{A}'-\mathcal{Z}')$. By Proposition \ref{r1}, there exists a partition {\boldmath$\mathcal{P}$} of $V$ such that, for {\boldmath$\mathcal{K}$} $=\{e \in \mathcal{E}(\mathcal{P}):\mathcal{Z}'\cap \mathcal{A}_e \neq \emptyset\}$,  we have $r_{{\sf M}_1}(\mathcal{Z}')=|\mathcal{Z}'\cap \mathcal{A}(\mathcal{P})|+|\mathcal{K}|+k(|V|-|\mathcal{P}|)$. By subcardinality and monotonicity of $r_{{\sf M}_2},$ we have $|\mathcal{Z}'\cap \mathcal{A}(\mathcal{P})|+r_{{\sf M}_2}(\mathcal{A}'-\mathcal{Z}')\ge r_{{\sf M}_2}(\mathcal{A}'-(\mathcal{Z}'-\mathcal{A}(\mathcal{P})))\ge r_{{\sf M}_2}(\mathcal{A}(\mathcal{P})\cup \mathcal{A}_{\mathcal{E}(\mathcal{P})-\mathcal{K}}).$

 The above three inequalities yield
\begin{align*}
k(|\mathcal{P}|-1)&\geq r_{{\sf M}_1}(\mathcal{Z}')+r_{{\sf M}_2}(\mathcal{A}'-\mathcal{Z}')-k(|V|-|\mathcal{P}|)\\
&=|\mathcal{Z}' \cap \mathcal{A}(\mathcal{P})|+|\mathcal{K}|+r_{{\sf M}_2}(\mathcal{A}'-\mathcal{Z}')\\
&\geq |\mathcal{K}|+r_{{\sf M}_2}(\mathcal{\mathcal{A}(\mathcal{P})\cup \mathcal{A}_{\mathcal{E}(\mathcal{P})-\mathcal{K}}}).
\end{align*}

\end{proof}

Let {\boldmath$\mathcal{P}$} be the partition of $V$ and {\boldmath$\mathcal{K}$} the hyperedge set from Claim \ref{part1} and let {\boldmath$\mathcal{Z}$}$=\mathcal{\mathcal{A}(\mathcal{P})\cup \mathcal{A}_{\mathcal{E}(\mathcal{P})-\mathcal{K}}}$. For some $X \in \mathcal{P}$, a dyperedge $a \in \mathcal{A}\cup \mathcal{A}_{\mathcal{E}-\mathcal{K}}$ contributes to either of $\sum_{v \in X}d_\mathcal{Z}^-(v)$ and $d_{\mathcal{A}\cup \mathcal{A}_{\mathcal{E}-\mathcal{K}}}^-(X)$ if and only if $head(a) \in X$ and $tail(a)-X \neq \emptyset$. This yields
\begin{equation}\label{ekd2}
\sum_{v \in X}d_\mathcal{Z}^-(v)=d_{\mathcal{A}\cup \mathcal{A}_{\mathcal{E}-\mathcal{K}}}^-(X).
\end{equation}

For all $\mathcal{P}'\subseteq \mathcal{P}$, as every hyperedge in $\mathcal{E}(\mathcal{P}')$ contributes to one of $|\mathcal{K}|$ and $\sum_{X\in\mathcal{P}'}d_{ \mathcal{A}_{\mathcal{E}-\mathcal{K}}}^-(X)$, we have 
\begin{equation}\label{ekd}
|\mathcal{K}|+\sum_{X\in\mathcal{P}'}d_{\mathcal{A}\cup \mathcal{A}_{\mathcal{E}-\mathcal{K}}}^-(X)\ge e_{\mathcal{E}\cup \mathcal{A}}(\mathcal{P}').
\end{equation}

Using Corollary \ref{gpmd}, we distinguish two cases depending on where the rank of $\mathcal{Z}$ in ${{\sf M}_2}$ is attained.

\begin{case}
$r_{{\sf M}_2}(\mathcal{Z})=\sum_{v \in V}\min \{k-f(v),d_\mathcal{Z}^-(v)\}.$
\end{case}

Let {\boldmath$\mathcal{P}'$} $=\{X\in\mathcal{P}:d_\mathcal{Z}^-(v)\leq k-f(v)$ for all $v \in X\}$. 
For all $X\in\mathcal{P}'$, by the definition of $\mathcal{P}'$ and \eqref{ekd2}, we have $\sum_{v \in X}\min\{k-f(v),d_\mathcal{Z}^-(v)\}=\sum_{v \in X}d_\mathcal{Z}^-(v)=d_{\mathcal{A}\cup \mathcal{A}_{\mathcal{E}-\mathcal{K}}}^-(X).$ 
For all $X\in\mathcal{P}-\mathcal{P}'$, there exists a vertex $v_X\in X$ with $k-f(v_X)<d_\mathcal{Z}^-(v_X)$, and then, by $f, k-f,d_\mathcal{Z}^-\geq 0$, we have $k-f(X)\le k-f(v_X)=\min\{k-f(v_X),d_\mathcal{Z}^-(v_X)\}\le\sum_{v \in X}\min\{k-f(v),d_\mathcal{Z}^-(v)\}.$

  Then, by  \eqref{pkz}, the case distinction made, \eqref{ekd} and \eqref{dktgy} for $\mathcal{P}'$, we obtain 
\begin{align*}
k(|\mathcal{P}|-1)&>|\mathcal{K}|+r_{{\sf M}_2}(\mathcal{\mathcal{A}(\mathcal{P})\cup \mathcal{A}_{\mathcal{E}(\mathcal{P})-\mathcal{K}}})\\
&=|\mathcal{K}|+\sum_{v \in V}\min \{k-f(v),d_\mathcal{Z}^-(v)\}\\
&=|\mathcal{K}|+\sum_{X\in \mathcal{P}'}\sum_{v \in X}\min \{k-f(v),d_\mathcal{Z}^-(v)\}+\sum_{X\in \mathcal{P}-\mathcal{P}'}\sum_{v \in X}\min \{k-f(v),d_\mathcal{Z}^-(v)\}\\
&\geq |\mathcal{K}|+\sum_{X\in \mathcal{P}'}d_{\mathcal{A}\cup \mathcal{A}_{\mathcal{E}-\mathcal{K}}}^-(X)+\sum_{X\in \mathcal{P}-\mathcal{P}'}(k-f(X))\\
&\ge e_{\mathcal{E}\cup \mathcal{A}}(\mathcal{\mathcal{P}'})+k(|\mathcal{P}|-|\mathcal{P}'|)-f(V-\cup \mathcal{P}')\\
&\geq k(|\mathcal{P}|-|\mathcal{P}'|)+k(|\mathcal{P}'|-1)\\
&= k(|\mathcal{P}|-1),
\end{align*}
 a contradiction.

\begin{case}
$r_{{\sf M}_2}(\mathcal{Z})=k(|V|-1)-\sum_{v \in V}\max\{0,k-g(v)-d_\mathcal{Z}^-(v)\}.$
\end{case}
Let {\boldmath$\mathcal{P}''$} $=\{X\in\mathcal{P}:k-g(v)-d_\mathcal{Z}^-(v)\geq 0$ for all $v \in X\}$. 
For all $X\in\mathcal{P}''$, by the definition of $\mathcal{P}''$ and \eqref{ekd2}, we have $\sum_{v \in X}\max\{0,k-g(v)-d_\mathcal{Z}^-(v)\}=k|X|-g(X)-d_{\mathcal{A}\cup \mathcal{A}_{\mathcal{E}-\mathcal{K}}}^-(X).$ 
For all $X\in \mathcal{P}-\mathcal{P}''$, there exists a vertex $v_X \in X$ with $0>k-g(v_X)-d_\mathcal{Z}^-(v_X)$, and then, by $g,d^-_\mathcal{Z} \geq 0$, we have $\sum_{v \in X}\max\{0,k-g(v)-d_\mathcal{Z}^-(v)\}=\sum_{v \in X-v_X}\max\{0,k-g(v)-d_\mathcal{Z}^-(v)\}\le k(|X|-1).$
By \eqref{pkz} and the case distinction we made, we obtain $k(|\mathcal{P}|-1)>|\mathcal{K}|+r_{{\sf M}_2}(\mathcal{Z})=|\mathcal{K}|+k(|V|-1)-\sum_{v \in V}\max\{0,k-g(v)-d_\mathcal{Z}^-(v)\}.$

This yields
\begin{align*}
k(|V|-|\mathcal{P}|)+|\mathcal{K}|&<\sum_{v \in V}\max\{0,k-g(v)-d_\mathcal{Z}^-(v)\}\\
&= \sum_{X \in \mathcal{P}''}\sum_{v \in X}\max\{0,k-g(v)-d_\mathcal{Z}^-(v)\}+ \sum_{X \in \mathcal{P}-\mathcal{P}''}\sum_{v \in X}\max\{0,k-g(v)-d_\mathcal{Z}^-(v)\}\\
&\leq \sum_{X \in \mathcal{P}''}(k|X|-g(X)-d_{\mathcal{A}\cup \mathcal{A}_{\mathcal{E}-\mathcal{K}}}^-(X))+\sum_{X \in \mathcal{P}-\mathcal{P}''}k(|X|-1)\\
&=k(|V|-|\mathcal{P}-\mathcal{P}''|)-\sum_{X \in \mathcal{P}''}(g(X)+d_{\mathcal{A}\cup \mathcal{A}_{\mathcal{E}-\mathcal{K}}}^-(X)).
\end{align*}

We obtain by \eqref{ekd} and \eqref{dktgy2} for $\mathcal{P}''$ that
\begin{align*}
k|\mathcal{P}''|&>|\mathcal{K}|+\sum_{X \in \mathcal{P}''}d_{\mathcal{A}\cup \mathcal{A}_{\mathcal{E}-\mathcal{K}}}^-(X)+\sum_{X \in \mathcal{P}''}g(X)\\
&\geq e_{\mathcal{A}\cup \mathcal{E}}(\mathcal{P}'')+\sum_{X \in \mathcal{P}''}g(X)\\
&\geq k|\mathcal{P}''|,
\end{align*}

a contradiction.
\medskip

The case distinction is complete which finishes the proof of Theorem \ref{new}.
\end{proofof}

Finally, we deal with the algorithmic consequences of Theorem \ref{matroidflex} which are contained in Theorem \ref{algoflex}.

\begin{proofof}{Theorem \ref{algoflex}}

It can be checked efficiently whether \eqref{gpc1} and \eqref{gpc2} are  satisfied. If not, then, by Lemma \ref{condgenpart}, one of \eqref{fg}, \eqref{dktgy} and \eqref{dktgy2} is not satisfied. By Lemma \ref{necessity}, no $(k,f,g)$-flexible packing of mixed hyperarborescences exists in $\mathcal{F}.$ Otherwise, by 
Theorem \ref{matroidflex}, the common independent sets of ${\sf M}^k_{\mathcal{F}}$ and ${\sf M}_{\mathcal{D}_{\mathcal{F}}}^{(k,f,g)}$ of size $k(|V|-1)$ are exactly the dyperedge sets of the $(k,f,g)$-feasible packings in $\mathcal{D}_{\mathcal{F}}$.

 Define {\boldmath $w'$}$:\mathcal{A}\cup \mathcal{A}_{\mathcal{E}}\rightarrow \mathbb{R}$ by $w'(a)=w(a)$ for all $a \in \mathcal{A}$ and $w'(a)=w(e)$ for all $a \in \mathcal{A}_{e}$ for all $e \in\mathcal{E}$. We first check if there is a common independent set of ${\sf M}^k_{\mathcal{F}}$ and ${\sf M}_{\mathcal{D}_{\mathcal{F}}}^{(k,f,g)}$ of size $k(|V|-1)$ and if this is the case, we find a common independent set of {\boldmath$\mathcal{A}^*$} of ${\sf M}^k_{\mathcal{F}}$ and ${\sf M}_{\mathcal{D}_{\mathcal{F}}}^{(k,f,g)}$ of size $k(|V|-1)$ minimizing $w'(\mathcal{A}^*)$. This can be done in polynomial time using Theorem \ref{matroidintersection}(b) because polynomial time independence oracles for ${\sf M}^k_{\mathcal{F}}$ and ${\sf M}_{\mathcal{D}_{\mathcal{F}}}^{(k,f,g)}$ are available by Lemmas \ref{obstacle} and \ref{oracle}. Now consider the subdypergraph {\boldmath$\mathcal{D}^*$} $=(V,\mathcal{A}^*)$ of $\mathcal{D}_{\mathcal{F}}$ and let {\boldmath$R$} be the multiset in $V$ in which every $v \in V$ is contained $k-d_{\mathcal{A}^*}^-(v)$ times. By Theorem $9(a)$ in \cite{fklst}, $\mathcal{A}^*$ can be decomposed into the dyperedge set of a packing of spanning hyperarborescences with root set $R$ in polynomial time. Replacing dyperedges in $\mathcal{A}_e$ by $e$ for all $e \in \mathcal{E}$, we obtain, by Lemma \ref{flexfeas}, the desired packing in $\mathcal{F}$. 
\end{proofof}
\bibliographystyle{siamplain}

\end{document}